\documentclass[12pt,letterpaper]{amsart}
\usepackage{fullpage}
\usepackage{amsmath}
\usepackage{amssymb}
\usepackage{amsthm}
\usepackage{color}
\usepackage{array}
\usepackage{xy}
\usepackage{setspace}
\usepackage{todonotes}   
\usepackage{multicol}
\usepackage{tikz}
\usepackage{comment}
\usepackage{hyperref}

\newtheoremstyle{slanted}
{3pt}
{3pt}
{\slshape}
{}
{\bfseries}
{.}
{.5em}
{}
\theoremstyle{slanted}
\newtheorem{theorem}{Theorem}[section]

\newtheorem{corollary}[theorem]{Corollary}

\newtheorem{definition}[theorem]{Definition}
\theoremstyle{remark}

\usepackage{enumitem}
\setlist[itemize]{leftmargin=*}
\setlist[enumerate]{leftmargin=*}
\usepackage{fancyhdr}
\begin{document}

\date{}
\author{Kevin Chang}
\title[Upper Bounds for Ordered Ramsey Numbers of Small 1-Orderings]{Upper Bounds for Ordered Ramsey Numbers of Small 1-Orderings}
\maketitle
\begin{abstract}
A $k$-ordering of a graph $G$ assigns distinct order-labels from the set $\{1,\ldots,|G|\}$ to $k$ vertices in $G$. Given a $k$-ordering $H$, the ordered Ramsey number $R_<(H)$ is the minimum $n$ such that every edge-2-coloring of the complete graph on the vertex set $\{1, \ldots, n\}$ contains a copy of $H$, the $i$th smallest vertex of which either has order-label $i$ in $H$ or no order-label in $H$.

This paper conducts the first systematic study of ordered Ramsey numbers for $1$-orderings of small graphs. We provide upper bounds for $R_<(H)$ for each connected $1$-ordering $H$ on $4$ vertices. Additionally, for every $1$-ordering $H$ of the $n$-vertex path $P_n$, we prove that $R_<(H) \in O(n)$. Finally, we provide an upper bound for the generalized ordered Ramsey number $R_<(K_n, H)$ which can be applied to any $k$-ordering $H$ containing some vertex with order-label $1$.
\end{abstract}

\section{Introduction}\label{sec:intro}

Ramsey's theorem establishes that for any graph $H$ and any sufficiently large $r$, every red-blue coloring of the complete graph on $r$ vertices must contain a monochromatic copy of $H$. The smallest $r$ for which this is the case is referred to as the Ramsey number $R(H)$. Loosely speaking, classical Ramsey theory focuses on bounding $R(G)$ both for specific small graphs \cite{ChHa72i} and for infinite families of large graphs \cite{ChHa72iii}. Both directions of research lead to notoriously difficult problems, the most famous of which is the study of $R(K_n)$. For specific small $n$, $R(K_n)$ is known only up to $n = 4$, with the best known bounds of $R(K_5)$ placing the value between $43$ and $49$ inclusive \cite{Exoo89, McRa95}. For arbitrary $n$, despite some smaller order improvements \cite{Conlon09, Spencer75}, the standard bounds of $2^{\frac{n}{2}} \le R(K_n) \le 2^{2n}$ \cite{Erdos47, ErSz35} have remained largely unchanged for sixty years \cite{CLFS14}.

This paper concerns ordered graphs. An \emph{ordered graph} $H$ is a graph whose vertices are assigned distinct order-labels from $\{1,2,\ldots,|H|\}$. Similarly, an ordered 2-coloring of the complete graph $K_n$ is a copy of $K_n$ in which every edge has been assigned one of two colors and the vertices have been labeled with $\{1, 2, \ldots, n\}$. Given an ordered 2-coloring $G$ and an ordered graph $H$, we say that $G$ \emph{contains a monochromatic copy} of $H$ if $G$ contains a subgraph isomorphic to $H$ whose edges are a single color and whose vertices are in the same relative order as in $H$. 

Interestingly, Ramsey's theorem generalizes to the context of ordered graphs and 2-colorings. In particular, for any ordered graph $H$, and a sufficiently large $r$, every ordered 2-coloring of the complete graph $K_r$ on $r$ vertices must contain a monochromatic copy of $H$. In recent years, this has prompted the study of so-called \emph{ordered Ramsey numbers}. Generalizing the Ramsey number, $R_<(H)$, the ordered Ramsey number of an ordered graph $H$, is the minimum number $n$ such that every ordered red-blue coloring of $K_n$ contains an ordered monochromatic copy of $H$. 

Ordered Ramsey numbers are a relatively new field of study. One of the oldest results, found by Erd\H{o}s and Szekeres in a classic 1935 paper, states that the ordered Ramsey number of the standard ordering of the $n$-vertex path is $(n-1)^2 + 1$ \cite{ErSz35}. In 2014, Conlon, Fox, Lee, and Sudakov initiated the systematic study of ordered Ramsey numbers. One of their main results highlights a surprising connection between ordered Ramsey numbers and their counterparts. In particular, there exists a constant $c$ such that the ordered Ramsey number of a graph $H$ is bounded above by $R(H)^{c \log^2 n}$ \cite{CLFS14}.

At roughly the same time, Balko, Cibulka, Kr{\'a}l, and Kyn{\v{c}}l pursued a related direction of study \cite{BCKK13}. Among other results, they found that ordered Ramsey numbers of the same graph can grow superpolynomially in the size of the graph in one ordering while remaining linear in another ordering. Considering ordered graphs satisfying restrictions, they found that graphs of constant bandwidth, interval chromatic number, and degeneracy have ordered Ramsey number polynomial in the number of vertices. In addition to examining ordered Ramsey numbers of arbitrary ordered graphs, they bounded the ordered Ramsey numbers of specific families, including paths, cycles, and stars.

So far, the study of ordered Ramsey numbers has focused on bounding the ordered Ramsey numbers of large families of graphs. Just as for classical Ramsey Numbers \cite{ErSz35, Radziszowski14, McRa95, Spencer75}, another natural problem is to study the ordered Ramsey numbers of particular small ordered graphs. 

This paper initiates the first systematic study of ordered Ramsey numbers of small graphs. In order to focus on orderings of particular importance, we introduce the notion of a \emph{1-ordering} of a graph, in which a single vertex of the graph is assigned a value from $\{1, \ldots, n\}$ indicating its value relative to the other vertices. The main focus of this paper is to find upper bounds for $R_<(H)$ for every 1-ordering $H$ of every connected graph on four vertices. A summary of these bounds can be found in Figure \ref{figtable} on page \pageref{figtable}. Perhaps surprisingly, several of the upper bounds in Figure \ref{figtable} are equal to the corresponding (unordered) Ramsey number, and are thus necessarily tight. These include all $1$-orderings of the 4-path $P_4$, three of the 1-orderings of the 4-star, and one $1$-ordering of the 3-pan.

Our results introduce two proof techniques which will likely have applicability in future work on small graphs. The first technique bounds $R_<(H)$ by finding an unordered graph $H'$ for which any ordering contains $H$ as an ordered subgraph; this establishes that $R_<(H) \le R(H')$, allowing us to harness already known results on the classical Ramsey number $R(H')$. The second technique uses what we call \emph{two-vertex anchoring}, in which one analyzes 2-colorings from the perspective of two preselected \emph{anchor vertices}. Two-vertex anchoring can be viewed as an ordered extension of classical Ramsey theory arguments in which one analyzes 2-colorings based on the edges coming out of a single vertex.

Additionally, in Section \ref{sec:infinite}, we are able to extend several of our results to consider ordered Ramsey number of infinite families of 1-orderings. We show that every 1-ordering of the $n$-vertex path has ordered Ramsey number at most linear in $n$. Moreover, we provide an upper bound for (the generalized Ramsey number) $R_n(K_n, H)$ for all graphs $H$ containing a vertex with order-label $1$. In particular, for any graph $H$ for which vertex $v$ has order-label $1$, we bound $R_<(K_n, H)$ in terms of $R_<(K_m, H \setminus \{v\})$ for $m \in [1, n]$.

The rest of this paper is organized as follows. In Section \ref{sec:prelim}, we introduce conventions and notation which will be used throughout the paper. In Section \ref{sec:smallgraphs}, we introduce our main results, establishing nontrivial upper bounds for ordered Ramsey numbers of 1-orderings of 4-vertex connected graphs. Then, in Section \ref{sec:infinite}, we provide extensions of results from Section \ref{sec:smallgraphs} which apply to infinite families of ordered graphs. Finally, in Section \ref{sec:conc}, we discuss directions of future work, including discussion of lower bounds for ordered Ramsey numbers of 1-orderings.

\section{Preliminaries}\label{sec:prelim}

In this section, we present notation and conventions. We begin by defining a $k$-ordered graph, also called a $k$-ordering.

\begin{definition}
A \emph{$k$-ordering} of a graph $H$ assigns distinct \emph{order-labels} from the set \\$\{1,\ldots,|H|\}$ to $k$ vertices in $H$.
\end{definition}

This paper focuses primarily on $1$-orderings, leading us to introduce additional notation:

\begin{definition}
Given a graph $H$ and a vertex $v \in H$, the \emph{$(v, l)$-ordering} of $H$ for some $l$ between 1 and $|H|$, inclusive, is the $1$-ordering of $H$ in which $v$ is assigned order-label $l$. 
\end{definition}

Figure \ref{figvertexlabelings} contains each of the 4-vertex connected graphs, the graph's name, and canonical names of each of its vertices. It is important to note that, as a convention, our proofs will use these canonical names when referring to vertices in each of the graphs.

\begin{figure}
\begin{tabular}{c c c}
\begin{tikzpicture}
\tikzstyle{vertex}=[circle,fill=black!25,minimum size=15pt,inner sep=0pt];
\tikzstyle{edge} = [draw,line width = 2 pt];
\foreach \pos/\name in {{(0,0)/d_2},{(.5,1)/d_1},{(1,0)/d_3},{(.5,-1)/d_4}} \node[vertex] (\name) at \pos {$\name$};
\foreach \source/ \dest in {d_1/d_2,d_3/d_2,d_4/d_2,d_4/d_2,d_3/d_1,d_4/d_3} \path[edge][black](\source) -- node[] {$$} (\dest);
\end{tikzpicture} &

\begin{tikzpicture}
\tikzstyle{vertex}=[circle,fill=black!25,minimum size=15pt,inner sep=0pt];
\tikzstyle{edge} = [draw,line width = 2 pt];
\foreach \pos/\name in {{(0,0)/p_1},{(1,0)/p_2},{(2,0)/p_3},{(3,0)/p_4}} \node[vertex] (\name) at \pos {$\name$};
\foreach \source/ \dest in {p_2/p_1,p_3/p_2,p_4/p_3} \path[edge][black](\source) -- node[] {$$} (\dest);
\end{tikzpicture} &

\begin{tikzpicture}
\tikzstyle{vertex}=[circle,fill=black!25,minimum size=15pt,inner sep=0pt];
\tikzstyle{edge} = [draw,line width = 2 pt];
\foreach \pos/\name in {{(0,0)/e_1},{(1,0)/e_2},{(2,.7)/e_3},{(2,-.7)/e_4}} \node[vertex] (\name) at \pos {$\name$};
\foreach \source/ \dest in {e_2/e_1,e_3/e_2,e_4/e_2,e_4/e_3} \path[edge][black](\source) -- node[] {$$} (\dest);
\end{tikzpicture} \\

Diamond graph & 4-path ($P_4$) & 3-pan \\

\begin{tikzpicture}
\tikzstyle{vertex}=[circle,fill=black!25,minimum size=15pt,inner sep=0pt];
\tikzstyle{edge} = [draw,line width = 2 pt];
\foreach \pos/\name in {{(0,0)/s_1},{(1,0)/s_2},{(2,.7)/s_3},{(2,-.7)/s_4}} \node[vertex] (\name) at \pos {$\name$};
\foreach \source/ \dest in {s_2/s_1,s_3/s_2,s_4/s_2} \path[edge][black](\source) -- node[] {$$} (\dest);
\end{tikzpicture} &

\begin{tikzpicture}
\tikzstyle{vertex}=[circle,fill=black!25,minimum size=15pt,inner sep=0pt];
\tikzstyle{edge} = [draw,line width = 2 pt];
\foreach \pos/\name in {{(0,0)/c_4},{(1,0)/c_3},{(1,1)/c_2},{(0,1)/c_1}} \node[vertex] (\name) at \pos {$\name$};
\foreach \source/ \dest in {c_2/c_1,c_3/c_2,c_4/c_1,c_4/c_3} \path[edge][black](\source) -- node[] {$$} (\dest);
\end{tikzpicture} &

\begin{tikzpicture}
\tikzstyle{vertex}=[circle,fill=black!25,minimum size=15pt,inner sep=0pt];
\tikzstyle{edge} = [draw,line width = 2 pt];
\foreach \pos/\name in {{(0,0)/k_4},{(1,0)/k_3},{(1,1)/k_2},{(0,1)/k_1}} \node[vertex] (\name) at \pos {$\name$};
\foreach \source/ \dest in {k_2/k_1,k_3/k_1,k_3/k_2,k_4/k_1,k_4/k_2,k_4/k_3} \path[edge][black](\source) -- node[] {$$} (\dest);
\end{tikzpicture} \\

4-star & 4-cycle ($C_4$) & 4-vertex complete graph ($K_4$)
\end{tabular}
\caption{Canonical vertex-labelings for each connected 4-vertex graph.}
\label{figvertexlabelings}
\end{figure}
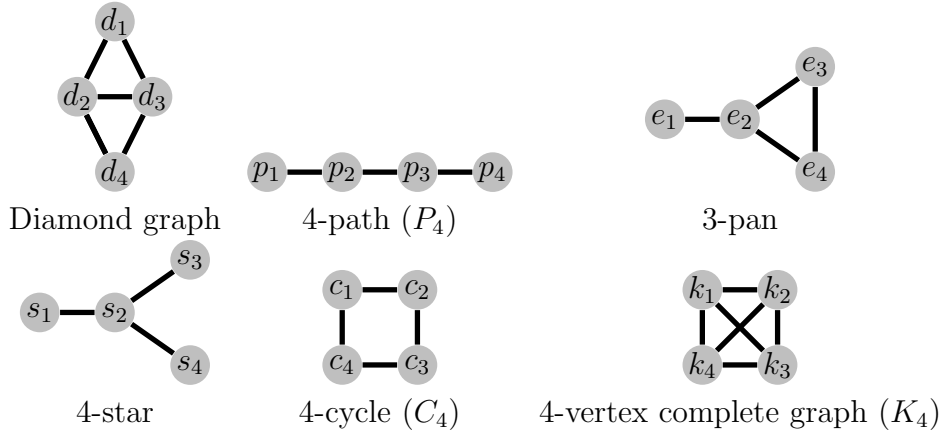

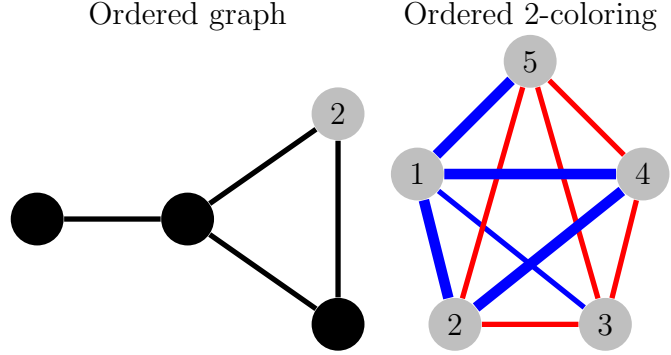
\begin{figure}
\tikzstyle{vertex}=[circle,fill=black!25,minimum size=20pt,inner sep=0pt]
\tikzstyle{edge} = [draw,line width = 2 pt]
\begin{tabular}[h]{c c}
Ordered graph & Ordered 2-coloring \\
\begin{tikzpicture}
\tikzstyle{vertex}=[circle,fill=black!25,minimum size=20pt,inner sep=0pt];
\tikzstyle{edge} = [draw,line width = 2 pt];
\foreach \pos/\name in {{(4,1.4)/2}} \node[vertex] (\name) at \pos {$\name$};
\tikzstyle{vertex}=[circle,fill=black,minimum size=20pt,inner sep=0pt];
\foreach \pos/\name in {{(0,0)/e_1},{(2,0)/e_2},{(4,-1.4)/e_4}} \node[vertex] (\name) at \pos {$\name$};
\foreach \source/ \dest in {e_2/e_1,2/e_2,e_4/e_2,e_4/2} \path[edge][black](\source) -- node[] {$$} (\dest);
\end{tikzpicture}
&
\begin{tikzpicture}
\foreach \pos/\name in {{(0,0)/2},{(-.5,2)/1},{(2.5,2)/4},{(2,0)/3},{(1,3.5)/5}}, \node[vertex] (\name) at \pos {$\name$};
\foreach \source/ \dest in {3/1} \path[edge][blue] (\source) -- node[] {$$} (\dest);
\foreach \source/ \dest in {3/2,4/3,5/2,5/3,5/4} \path[edge][red] (\source) -- node[] {$$} (\dest);
\tikzstyle{edge} = [draw,line width = 4 pt]
\foreach \source/ \dest in {2/1,4/2,4/1,5/1} \path[edge][blue]  (\source) -- node[] {$$} (\dest);
\tikzstyle{edge} = [draw,line width = 2 pt]
\end{tikzpicture}
\end{tabular}
\caption{On the left is an example of a $1$-ordering of a graph on 4 vertices. On the right is an example of an ordered 2-coloring containing a monochromatic copy of the $1$-ordering (bolded in blue).}
\label{figex1}
\end{figure}

The left side of Figure \ref{figex1} provides an example of a $1$-ordering of the 3-pan graph. Note that this is the $(e_3, 2)$-ordering of the graph.

Next we define the notion of an ordered $2$-coloring and what it means for an ordered 2-coloring to contain a copy of a $k$-ordering.

\begin{definition}
A \emph{2-coloring} on $n$ vertices is a set of $n$ vertices between every two of which is either a red or a blue edge. Additionally, the 2-coloring is said to be \emph{ordered} if every vertex has a unique integer from $1$ to $n$ assigned to it, indicating the order-label of the vertex.
\end{definition}

\begin{definition}
Given an ordered 2-coloring $G$ and a $k$-ordering $H$, we say that $G$ \emph{contains} $H$ if there exists a copy of $H$ in $G$ that satisfies the following properties:
\begin{enumerate}
\item The edges in the copy of $H$ are all the same color. (monochromatic)

\item The $i$th smallest-labeled vertex in the copy of $H$ either has order-label $i$ in $H$ or no order-label in $H$. (order-preserving)

\end{enumerate}
\end{definition}

\begin{definition}
Given an ordered 2-coloring $G$ and a $k$-coloring $H$, we say that $G$ \emph{avoids} $H$ if $G$ does not contain $H$.
\end{definition}

The right side of Figure \ref{figex1} provides an example of an ordered 2-coloring containing a $1$-ordering. In particular, vertices $5, 1, 2, 4$ in the 2-coloring correspond with vertices $e_1$, $e_2$, $e_3$, and $e_4$ in the $1$-ordering. 

Note that, as a convention, we will often refer to vertices in ordered 2-colorings by their order-label. In contrast, vertices in $k$-orderings are normally assigned names independently of their order-labels. Additionally, when we talk about vertices forming a copy of a graph, we will always list the vertices in the same order as in we list vertices in Figure \ref{figvertexlabelings}.

Finally, we define the generalized ordered Ramsey number.

\begin{definition}
The \emph{ordered Ramsey number} $R_<(H_1,H_2)$ of a $k$-ordering $H_1$ and a $j$-ordering $H_2$ is the smallest number $n$ such that any ordered 2-coloring on $n$ vertices must contain a red copy of $H_1$ or a blue copy of $H_2$. As a shorthand, $R_<(H, H)$ is denoted by $R_<(H)$.
\end{definition}

\section{Small Graphs}\label{sec:smallgraphs}

\begin{figure}
\begin{tabular}[h]{| c  c |}
\hline
Graph &
\begin{tabular}{|m{4cm} | m{3cm} | m{4cm}}
Ordering & Upper Bound & Original Source \\ [2ex]
\end{tabular} \\ [1ex] \hline

\hline
\begin{tikzpicture}[baseline={([yshift=-.5ex]current bounding box.center)}]
\tikzstyle{vertex}=[circle,fill=black!25,minimum size=15pt,inner sep=0pt];
\tikzstyle{edge} = [draw,line width = 2 pt];
\foreach \pos/\name in {{(0,0)/d_2},{(.5,1)/d_1},{(1,0)/d_3},{(.5,-1)/d_4}} \node[vertex] (\name) at \pos {$\name$};
\foreach \source/ \dest in {d_1/d_2,d_3/d_2,d_4/d_2,d_4/d_2,d_3/d_1,d_4/d_3} \path[edge][black](\source) -- node[] {$$} (\dest);
\end{tikzpicture} &

\begin{tabular}{|m{4cm} | m{3cm} |m{4cm}}
$(d_1, 1)$-ordering & 14 & Theorem \ref{dgva1} \\ [2ex] \hline
$(d_1, 2)$-ordering & 16 & Theorem \ref{dgva2} \\ [2ex] \hline
$(d_2, 1)$-ordering & 13 & Theorem \ref{dgvb1} \\ [2ex] \hline
$(d_2, 2)$-ordering & 17 & Theorem \ref{dgvb2} \\ [2ex]
\end{tabular}

\\ [8ex] \hline

\begin{tikzpicture}[baseline={([yshift=-.5ex]current bounding box.center)}]
\tikzstyle{vertex}=[circle,fill=black!25,minimum size=15pt,inner sep=0pt];
\tikzstyle{edge} = [draw,line width = 2 pt];
\foreach \pos/\name in {{(0,0)/p_1},{(1,0)/p_2},{(2,0)/p_3},{(3,0)/p_4}} \node[vertex] (\name) at \pos {$\name$};
\foreach \source/ \dest in {p_2/p_1,p_3/p_2,p_4/p_3} \path[edge][black](\source) -- node[] {$$} (\dest);
\end{tikzpicture} &

\begin{tabular}{|m{4cm} | m{3cm} |m{4cm}}
$(p_1, 1)$-ordering & 5 & Theorem \ref{4p} \\  [2ex] \hline
$(p_1, 2)$-ordering & 5 & Theorem \ref{4p} \\  [2ex] \hline
$(p_2, 1)$-ordering & 5 & Theorem \ref{4p} \\  [2ex] \hline
$(p_2, 2)$-ordering & 5 & Theorem \ref{4p} \\[2ex]
\end{tabular}

 \\[8ex] \hline

\begin{tikzpicture}[baseline={([yshift=-.5ex]current bounding box.center)}]
\tikzstyle{vertex}=[circle,fill=black!25,minimum size=15pt,inner sep=0pt];
\tikzstyle{edge} = [draw,line width = 2 pt];
\foreach \pos/\name in {{(0,0)/e_1},{(1,0)/e_2},{(2,.7)/e_3},{(2,-.7)/e_4}} \node[vertex] (\name) at \pos {$\name$};
\foreach \source/ \dest in {e_2/e_1,e_3/e_2,e_4/e_2,e_4/e_3} \path[edge][black](\source) -- node[] {$$} (\dest);
\end{tikzpicture} &

\begin{tabular}{|m{4cm} | m{3cm} |m{4cm}}
$(e_1, 1)$-ordering & 10 & Theorem \ref{3pva1} \\  [2ex] \hline
$(e_1, 2)$-ordering & 10 & Theorem \ref{3pva2} \\  [2ex] \hline
$(e_2, 1)$-ordering & 7 & Theorem \ref{3pvb1} \\  [2ex] \hline
$(e_2, 2)$-ordering & 10 & Theorem \ref{3pvb2} \\  [2ex] \hline
$(e_3, 1)$-ordering & 10 & Theorem \ref{3pvc} \\  [2ex] \hline
$(e_3, 2)$-ordering & 9 & Theorem \ref{3pvc2} \\[2ex]
\end{tabular} \\[12ex] \hline

\begin{tikzpicture}[baseline={([yshift=-.5ex]current bounding box.center)}]
\tikzstyle{vertex}=[circle,fill=black!25,minimum size=15pt,inner sep=0pt];
\tikzstyle{edge} = [draw,line width = 2 pt];
\foreach \pos/\name in {{(0,0)/s_1},{(1,0)/s_2},{(2,.7)/s_3},{(2,-.7)/s_4}} \node[vertex] (\name) at \pos {$\name$};
\foreach \source/ \dest in {s_2/s_1,s_3/s_2,s_4/s_2} \path[edge][black](\source) -- node[] {$$} (\dest);
\end{tikzpicture} &

\begin{tabular}{|m{4cm} | m{3cm} |m{4cm}}
$(s_1, 1)$-ordering & 6 & Theorem \ref{4sva1} \\  [2ex] \hline
$(s_1, 2)$-ordering & 6 & Theorem \ref{4sva2} \\  [2ex] \hline
$(s_2, 1)$-ordering & 6 & Theorem 12 of \cite{BCKK13} \\  [2ex] \hline
$(s_2, 2)$-ordering & 9 & Theorem 12 of \cite{BCKK13} \\[2ex]
\end{tabular} \\[8ex] \hline

\begin{tikzpicture}[baseline={([yshift=-.5ex]current bounding box.center)}]
\tikzstyle{vertex}=[circle,fill=black!25,minimum size=15pt,inner sep=0pt];
\tikzstyle{edge} = [draw,line width = 2 pt];
\foreach \pos/\name in {{(0,0)/c_4},{(1,0)/c_3},{(1,1)/c_2},{(0,1)/c_1}} \node[vertex] (\name) at \pos {$\name$};
\foreach \source/ \dest in {c_2/c_1,c_3/c_2,c_4/c_1,c_4/c_3} \path[edge][black](\source) -- node[] {$$} (\dest);
\end{tikzpicture} &

\begin{tabular}{|m{4cm} | m{3cm} |m{4cm}}
$(c_1, 1)$-ordering & 6 & Page 391 of \cite{ChHa72ii} \\ [3ex] \hline
$(c_1, 2)$-ordering & 6 & Page 391 of \cite{ChHa72ii} \\ [3ex]
\end{tabular}
\\[1ex] \hline 

\begin{tikzpicture}[baseline={([yshift=-.5ex]current bounding box.center)}]
\tikzstyle{vertex}=[circle,fill=black!25,minimum size=15pt,inner sep=0pt];
\tikzstyle{edge} = [draw,line width = 2 pt];
\foreach \pos/\name in {{(0,0)/k_4},{(1,0)/k_3},{(1,1)/k_2},{(0,1)/k_1}} \node[vertex] (\name) at \pos {$\name$};
\foreach \source/ \dest in {k_2/k_1,k_3/k_1,k_3/k_2,k_4/k_1,k_4/k_2,k_4/k_3} \path[edge][black](\source) -- node[] {$$} (\dest);
\end{tikzpicture} &

\begin{tabular}{|m{4cm} | m{3cm} |m{4cm}}
$(k_1, 1)$-ordering & 18 & Page 391 of \cite{ChHa72ii} \\ [3ex] \hline
$(k_1, 2)$-ordering & 18 & Page 391 of \cite{ChHa72ii} \\ [3ex]
\end{tabular}\\[1ex]\hline  

\end{tabular}
\caption{Upper bounds on ordered Ramsey numbers of 4-vertex connected 1-orderings.}
\label{figtable}
\end{figure}

In this section we present our main results, providing an upper bound for $R_<(H)$ for each 1-ordering $H$ of each connected 4-vertex graph. These results are summarized in Figure \ref{figtable}, which provides for each $1$-ordering the best known upper bound for $R_<(H)$, as well as a citation to the proof of that bound. Note that although each 4-vertex graph has four $1$-orderings for each of its four nodes, $1$-orderings with order-label $k$ and those with order-label $(5 - k)$ are equivalent by symmetry; moreover, reflectional and rotational vertex symmetries additionally render certain $1$-orderings equivalent. To account for this, Figure \ref{figtable} only considers a single $1$-ordering in each symmetry class. Additionally, two graphs in Figure \ref{figtable} have upper bounds which are trivially obtained: for both the complete graph $K_4$ and the cycle $C_4$, any copy of the graph is guaranteed to contain every $1$-ordering of it, rendering the ordered and unordered Ramsey numbers equal.

Our proofs introduce two techniques. The first, which we call the \emph{parent-graph technique} is to exploit the existence of another graph that is guaranteed to contain a copy of our ordered graph. The second, which we call the \emph{two-vertex anchoring technique} is to analyze ordered 2-colorings from the perspective of two anchor vertices. This is a natural extension of small-graph Ramsey number arguments, which typically analyze 2-colorings from the perspective of a single vertex.

We begin by applying the parent-graph technique to 1-orderings of the 4-path $P_4$. Surprisingly, we find that the ordered Ramsey number of each 1-ordering is precisely the same as $R(P_4)$.

\begin{theorem}
Let $H$ be any 1-ordering of the 4-path $P_4$. Then $R_<(H) = 5$.
\label{4p}
\end{theorem}

\begin{proof}
Let $i$ be such that the $i$th vertex of $H$ has an order-label, and denote that order-label $l$. Without loss of generality, $l$ is either $1$ or $2$.

We begin by observing that any ordering of the graph $C_4$ (with vertices $c_1, c_2, c_3, c_4$) must contain $H$ as a subgraph. In particular, because $C_4$ is rotationally symmetric, we may assume without loss of generality that $c_i$ has order-label $l$. It follows that vertices $c_1, c_2, c_3, c_4$ form a copy of $H$.

Additionally, we observe that any ordering of the graph $C_5$ (with vertices $x_1, x_2, x_3, x_4, x_5$) must contain $H$ as a subgraph. Because $C_5$ is symmetric, without loss of generality, $x_i$ is a vertex with order-label $l$. If $l = 1$, vertices $x_1, x_2, x_3, x_4$ form a copy of $H$. On the other hand, if $l = 2$, and $x_5$ does not have order-label $1$, then we are also done since vertices $x_1, x_2, x_3, x_4$ form a copy of $H$. Finally, if $l = 2$ and $x_5$ has order-label $1$, then consider the path $P$ made by vertices $x_{2i - 1} , x_{2i - 2}, x_{2i - 3}, x_{2i - 4}$, with indices taken modulo $5$. The path $P$ must contain $x_5$ (since $i \neq 5$ means that $2i \pmod 5 \neq 5$) and contains $x_i$ as its $i$th vertex, meaning that $P$ forms a copy of $H$.

Let $G$ be an ordered 2-coloring avoiding $H$, and for the sake of contradiction, suppose that $|G| = 5$. If $G$ does not contain a monochromatic triangle, then $G$ must consist of a red copy of $C_5$ and a blue copy of $C_5$ (Page 390 of \cite{ChHa72ii}). However, the presence of a monochromatic $C_5$ implies that $G$ contains $H$, a contradiction. Thus, we assume that $G$ contains a monochromatic triangle. Without loss of generality, this triangle, formed by vertices $u_1,u_2,u_3$, is red. Let $w_1,w_2$ be the two vertices not contained within this triangle. If there exist at least 2 red edges between $w_1$ and the vertices in the triangle ($u_1,u_2$ without loss of generality), then vertices $w_1,u_1,u_3,u_2$ form a red copy of $C_4$; thus, $G$ is forced to contain $H$, a contradiction. Therefore, for each $i$, the edge $(w_i,u_j)$ can be red for at most one $j$.

We claim that for each $i$, there exists a unique $j$ such that the edge $(w_i,u_j)$ is red. For the sake of contradiction, suppose that the edges $(w_1,u_j)$ are blue for all $j$. Without loss of generality, edges $(w_2,u_1)$ and $(w_2,u_2)$ are also blue. Then vertices $w_1,u_1,w_2,u_2$, in that order, form a blue copy of $C_4$, thus containing a copy of $H$, a contradiction. 

Without loss of generality, edges $(w_1,u_1)$ and $(w_2,u_2)$ are red, and all other edges of the form $(w_i,u_j)$ are blue. If edge $(w_1,w_2)$ is red, then vertices $w_1,u_1,u_2,w_2$ form a red copy of $C_4$, a contradiction. Thus, edge $(w_1,w_2)$ must be blue.

All edge colors in $G$ are now determined. In particular, $(w_1, w_2)$, $(w_1, u_2)$, $(w_1, u_3)$, $(w_2, u_1)$, and $(w_2, u_3)$ are blue, while the remaining edges are red. 

Performing casework on the location of the smallest vertex of $G$ (which we omit for the sake of brevity), we can see that any 1-ordering of $P_4$ with order-label $l = 1$ is contained in $G$. Performing casework on the locations of the smallest and second-smallest vertices of $G$, we can see that any 1-ordering of $P_4$ with order-label $l = 2$ is contained in $G$. By symmetry, $G$ must contain all other 1-orderings of the 4-path.

Hence, in all cases, if $|G| = 5$, then $G$ must contain $H$. This implies that $R_<(H) \le 5$. Since $R(P_4) = 5$ \cite{ChHa72ii} and $H$ is a 1-ordering of $P_4$, we must have $R_<(H) = 5$.


\end{proof}

Next, we apply the parent-graph technique to certain $1$-orderings of the 3-pan. Later, in Theorem \ref{3pvc2}, we will strengthen this bound for $l$ in the set $\{2, 3\}$.

\begin{theorem}
For any $l \in \{1, 2, 3, 4\}$, let $H$ be the $(e_3,l)$-ordering of the 3-pan. Then $R_<(H) \le 10$.
\label{3pvc}
\end{theorem}

\begin{proof}
Let $G$ be a 2-coloring containing a monochromatic copy of the diamond graph. We claim that $G$ must also contain $H$. In particular, it suffices to show that any ordering of the diamond graph must contain a copy of $H$. If $d_1$ in the diamond graph has order-label $l$, vertices $d_1$, $d_2$, $d_3$, and $d_4$ in the diamond graph correspond to vertices $e_3$, $e_2$, $e_4$, and $e_1$, respectively, in the 3-pan. If $d_2$ in the diamond graph has order-label $l$, vertices $d_1$, $d_2$, $d_3$, and $d_4$ correspond to vertices $e_1$, $e_3$, $e_2$, and $e_4$, respectively, in the 3-pan. Thus, any $e_3$-ordering of the 3-pan is contained within the diamond graph. Since the Ramsey number of the diamond graph is 10 \cite{ChHa72ii}, we have $R_<(H) \le 10$.
\end{proof}

For completeness, we also apply the parent-graph technique to a 1-ordering of the 4-star. This application is somewhat less interesting than the others.

\begin{theorem}
Let $H$ be the $(s_1,2)$-ordering of the 4-star. Then $R_<(H) = 6$.
\label{4sva2}
\end{theorem}

\begin{proof}
Note that any copy of a $(s_2,1)$-ordering of the 4-star is also a copy of a $(s_1,2)$-ordering of the 4-star, since vertices $s_1,s_3,s_4$ are symmetric. Since the ordered Ramsey number of the $(s_2,1)$-ordering of the 4-star is 6 (Theorem 12 of \cite{BCKK13}), we have $R_<(H) \le 6$. Moreover, since the Ramsey number of the 4-star is 6 (Page 391 of \cite{ChHa72ii}), $R_<(H) \ge 6$, implying equality.
\end{proof}

Our next seven results, the paper's most difficult, rely on the two-vertex anchoring technique. Classically, one finds upper bounds for Ramsey numbers of small graphs using what we call \emph{single-vertex anchoring}, in which the proof centers around a single vertex. In particular, one selects a vertex $v$ in the 2-coloring and individually analyzes the sets $X$ of vertices connected to $v$ by red edges and $Y$ of edges connected to $v$ by blue edges. On the other hand, two-vertex anchoring selects two vertices $v_1$ and $v_2$ with sets $X_1, Y_1$ and $X_2, Y_2$ similarly defined for $v_1$ and $v_2$, respectively. The proof proceeds by individually analyzing $X_1 \cap X_2$, $X_1 \cap Y_2$, $Y_1 \cap X_2$, and $Y_1 \cap Y_2$.


Our next theorem uses two-vertex anchoring to prove an upper bound on the ordered Ramsey number of the remaining 1-ordering of the 4-star.

\begin{theorem}
Let $H$ be the $(s_1,1)$-ordering of the 4-star. Then $R_<(H) = 6$.
\label{4sva1}
\end{theorem}

\begin{proof}
Let $G$ be an ordered 2-coloring avoiding $H$. Without loss of generality, edge $(1,2)$ in $G$ is blue. Let $V_R$ be the set of vertices connected to vertex 1 by red edges, and let $V_B$ be the set of vertices other than vertex 2 connected to vertex 1 by blue edges.

Note that there can be at most 1 blue edge between vertex 2 and a vertex with a larger order-label. Otherwise, there would exist a blue copy of $H$ formed by vertices 1 and 2 and any two other vertices connected to vertex 2 by blue edges. Let $V_{RR}$ be the set of vertices in $V_R$ connected to vertex 2 by red edges, and let $V_{BR}$ be the set of vertices in $V_B$ connected to vertex 2 by blue edges. We have already shown that there are at most 3 vertices outside $V_{RR} \cup V_{BR}$. We finish by considering the following cases:
\begin{itemize}
\item $|V_{RR}| = 0$: Note that $V_{BR}$ must avoid monochromatic 3-paths, since combined with either vertex 1 or 2, a monochromatic 3-path in $V_{BR}$ would yield a monochromatic copy of $H$. Since the Ramsey number of the 3-path is 3, $V_{BR}$ has size at most 2. Thus, $|G| \le 3 + |V_{BR}| \le 3 + 2 = 5$.

\item $|V_{RR}| \ge 1, |V_{BR}| = 0$: Note that $V_{RR}$ can only contain blue edges, since if $V_{RR}$ contained a red edge, there would exist a red copy of $H$ formed by vertices 1 and 2 and the two vertices in the red edge. If $|V_{RR}| \ge 4$, then there exists a blue copy of $K_4$ and thus a blue copy of $H$ in $G$, a contradiction. Thus, $|V_{RR}| \le 3$. If there does not exist a vertex $u > 2$ such that edge $(2,u)$ is blue, then we are done, as $|G| = 2 + |V_{RR}| \le 2 + 3 \le 5$.

Now suppose that there exists a vertex $u > 2$ with edge $(2,u)$ blue. For the sake of contradiction, assume that $|V_{RR}| \ge 3$. Let $v_1,v_2,v_3 \in V_{RR}$ be such that $v_1 < v_2 < v_3$. If edge $(u,v_3)$ is blue, then the 4-star formed by vertices $u,v_1,v_2,v_3$, with $v_3$ in the center of the star, is blue since $V_{RR}$ contains only blue edges. Moreover, since $v_3 > v_1, v_2$, the 4-star forms a copy of $H$, the $(s_1, 1)$-ordering of the 4-star, a contradiction. Since edge $(u,v_3)$ is red, there exists a red copy of $H$ formed by vertices $1,v_3,2,u$, a contradiction. Thus, $|V_{RR}| \le 2$, and $|G| = 3 + |V_{RR}| \le 3 + 2 = 5$.

\item $|V_{RR}|, |V_{BR}| \ge 1$: Let $v \in V_{BR}$. If there exist $v_1,v_2 > 1$ such that edges $(v,v_1)$ and $(v,v_2)$ are blue, then since edge $(1,v)$ is blue, vertices $1,v,v_1,v_2$ form a blue copy of $H$. If there exist $v_1,v_2 > 2$ such that edges $(v,v_1)$ and $(v,v_2)$ are red, then since edge $(2,v)$ is red, vertices $1,v,v_1,v_2$ form a red copy of $H$. This implies that there are at most two blue edges and at most two red edges containing $v$. Thus, $|G| \le 1 + 2 + 2 = 5$.
\end{itemize}
\end{proof}

Theorems \ref{3pva1} and \ref{3pvc2} use two-vertex anchoring to prove upper bounds on the ordered Ramsey numbers of two 1-orderings of the 3-pan.

\begin{theorem}
Let $H$ be the $(e_1,1)$-ordering of the 3-pan. Then $R_<(H) \le 10$.
\label{3pva1}
\end{theorem}

\begin{proof}
Let $G$ be an ordered 2-coloring avoiding $H$. Define $V_{RR}, V_{RB}, V_{BB}, V_{BR}$ to be the sets of vertices in $G$ connecting to vertex 1 by red, red, blue, and blue edges, respectively, and to vertex 2 by red, blue, blue, and red edges, respectively. Without loss of generality, the edge between vertices 1 and 2 is red.

Observe that $V_{BB}$ contains only red edges, since if there were a blue edge $(u, v)$ in $V_{BB}$, there would be a blue copy of $H$ formed by vertices $1$, $u$, $2$, $v$. Thus, $|V_{BB}| \le 3$. Also note that $V_{RR}$ and $V_{BR}$ can only contain blue edges, since if there were a red edge $(u,v)$ in either set, a red copy of $H$ would be formed by vertices $1$, $2$, $u$, $v$. (Recall that the edge between 1 and 2 is red.) This implies that $|V_{RR}| \le 3$ and $|V_{BR}| \le 3$. 

Moreover, all edges between vertices in $V_{BB}$ and vertices in $V_{RB}$ must be red and all edges between vertices in $V_{RR}$ and vertices in $V_{BR}$ must be blue. In particular, suppose that there is a blue edge $(u,v)$ from $V_{RB}$ to $V_{BB}$. Then a blue copy of $H$ can be constructed using vertices $1$, $v$, $2$, $u$. Similarly, if there is a red edge $(u, v)$ from $V_{RR}$ to $V_{BR}$, then a red copy of $H$ can be constructed using vertices $1$, $v$, $2$, $u$.

Additionally, observe that any monochromatic triangle in $V_{RB} \cup V_{BR}$ would form a copy of $H$ with one of vertices 1 or 2. Since $R(K_3) = 6$, we have $|V_{RB}| + |V_{BR}| \le 5$.

Aided by the above observations, we use the following cases to establish that that $|G| \le 9$:
\begin{itemize}
\item $|V_{RB}| \ge 1, |V_{BR}| \ge 1$: Let $b,d$ be vertices such that $b \in V_{RB}$ and $d \in V_{BR}$. Since all edges between $V_{BB}$ and $V_{RB}$ are red, if there exists a red edge $(v_1,v_2) \in V_{BB}$, then vertices $1,b,v_1,$ and $v_2$ form a red copy of $H$. Thus, $V_{BB}$ cannot have any red edges. Since $V_{BB}$ cannot have any blue edges either, $|V_{BB}| \le 1$. Similarly, if there exists a blue edge $(v_1,v_2) \in V_{RR}$, then vertices $1,d,v_1,$ and $v_2$ form a blue copy of $H$. This means that $V_{RR}$ cannot have any blue edges. Because $V_{RR}$ also avoids red edges, $|V_{RR}| \le 1$. Thus, $|G| = 2 + |V_{BB}| + |V_{RR}| + (|V_{RB}| + |V_{BR}|) \le 2 + 1 + 1 + 5 = 9$.

\item $|V_{RB}| \ge 1, |V_{BR}| = 0$: By the argument we used in the first case, we know that $|V_{BB}| \le 1$. We claim that $|V_{RB} \cup V_{RR}| \le 6$. For the sake of contradiction, suppose otherwise. Since $R(K_3) = 6$, there must exist a monochromatic triangle in $V_{RB} \cup V_{RR}$. One can verify that unless this triangle is a blue triangle entirely contained within $V_{RR}$, the triangle along with either vertex 1 or vertex 2 forms a monochromatic copy of $H$. If this triangle is a blue triangle entirely contained within $V_{RR}$, then by removing one of the vertices in this triangle, we are left with at least $6$ vertices in $V_{RB} \cup V_{RR}$ and therefore another monochromatic triangle. However, the new triangle cannot be a blue triangle entirely contained within $V_{RR}$ (since $|V_{RR}| \le 3$ and we have removed a vertex from $V_{RR}$), a contradiction. We have $|G| = 2 + |V_{BB}| + |V_{BR}| + (|V_{RB}| + |V_{RR}|) \le 2 + 1 + 0 + 6 = 9$.

\item $|V_{RB}| = 0, |V_{BR}| \ge 1$: By the argument used in the first case, we know that $|V_{RR}| \le 1$. Thus, $|G| = 2 + |V_{BB}| + |V_{RB}| + |V_{RR}| + |V_{BR}| \le 2 + 3 + 0 + 1 + 3 = 9$.

\item $|V_{RB}| = 0, |V_{BR}| = 0$: We have $|G| = 2 + |V_{BB}| + |V_{RB}| + |V_{RR}| + |V_{BR}| \le 2 + 3 + 0 + 3 + 0 = 8$.
\end{itemize}
\end{proof}

The following upper bound on the ordered Ramsey number of the $(e_3,2)$-ordering of the 3-pan actually combines two-vertex anchoring with Theorem \ref{3pvc}, an application of the parent-graph technique. In doing so, it strengthens the bound provided by Theorem \ref{3pvc}.

\begin{theorem}
Let $H$ be the $(e_3,2)$-ordering of the 3-pan. Then $R_<(H) \le 9$.
\label{3pvc2}
\end{theorem}

\begin{proof}
Let $G$ be an ordered 2-coloring on $n$ vertices avoiding $H$. Without loss of generality, edge $(1,n)$ is red. Define $V_{RR}, V_{RB}, V_{BB}, V_{BR}$ to be the sets of vertices in $G$ connecting to vertex 1 by red, red, blue, and blue edges and to vertex $n$ by red, blue, blue, and red edges, respectively.

Since any monochromatic copy of the diamond graph contains $H$ (Theorem \ref{3pvc}), $G$ cannot contain any monochromatic copies of the diamond graph. This implies that $V_{RR}$ contains at most one vertex, since any two vertices in $V_{RR}$ along with vertices 1 and $n$ would yield a monochromatic copy of the diamond graph. Moreover, if $V_{BB}$ contains a blue edge $(v_1,v_2)$, then vertices $1,v_1,v_2,n$ form a blue copy of the diamond graph. Thus, $V_{BB}$ contains only red edges. Note that if $V_{RB}$ contains a red edge $(v_1,v_2)$, then since $1 < v_1,v_2 < n$, vertices $n,1,v_1,v_2$ form a red copy of $H$ with $v_1$ or $v_2$ in position $e_3$; interestingly, this is the only step in this proof where we construct a copy of $H$ directly, instead of through a copy of the diamond graph. Thus, $V_{RB}$ contains only blue edges. If $|V_{RB}| \ge 3$, then three vertices in $V_{RB}$ along with vertex $n$ form a blue copy of $K_4$, a contradiction. Thus, $|V_{RB}| \le 2$.

A similar argument shows that $V_{BR}$ contains only blue edges. Since any three vertices in $V_{BR}$ along with vertex 1 form a blue copy of $K_4$, $|V_{BR}| \le 2$.

Consider the following cases:
\begin{itemize}
\item $|V_{RB}|,|V_{BR}| \le 1$: Since $V_{BB}$ contains only red edges, $V_{BB}$ must contain fewer than four vertices. Otherwise, $V_{BB}$ would contain a red copy of $K_4$, a contradiction. Since $|V_{BB}| \le 3$, we have $|G| = 2 + |V_{RR}| + |V_{RB}| + |V_{BB}| + |V_{BR}| \le 2 + 1 + 1 + 3 + 1 = 8$.

\item $|V_{RB}| = 2$: For the sake of contradiction, assume that $|V_{BB}| \ge 2$. Let $u_1,u_2 \in V_{RB}$ and $v_1,v_2 \in V_{BB}$. Note that edge $(u_1,u_2)$ is blue. If edge $(u_1,v_1)$ is blue, then vertices $v_2,v_1,n,u_1$ form a blue copy of the diamond graph, which must contain a blue copy of $H$, a contradiction. Thus, edge $(u_1,v_1)$ must be red, and by similar reasoning, edges $(u_1,v_2),(u_2,v_1),(u_2,v_2)$ must all be red. Since edge $(v_1,v_2)$ is red, vertices $u_1,v_1,v_2,u_2$ form a red copy of the diamond graph, a contradiction. Thus, $|V_{BB}| \le 1$. We have $|G| = 2 + |V_{RR}| + |V_{RB}| + |V_{BB}| + |V_{BR}| \le 2 + 1 + 2 + 1 + 2 = 8$.

\item $|V_{BR}| = 2$: Again, for the sake of contradiction, assume that $|V_{BB}| \ge 2$. By similar reasoning to the above case, there exists a monochromatic copy of the diamond graph, implying that $|V_{BB}| \le 1$. We have $|G| = 2 + |V_{RR}| + |V_{RB}| + |V_{BB}| + |V_{BR}| \le 2 + 1 + 2 + 1 + 2 = 8$.
\end{itemize}
\end{proof}

Theorems \ref{dgva1}, \ref{dgva2}, \ref{dgvb1}, and \ref{dgvb2} consider each of the four $1$-orderings of the diamond graph.

\begin{theorem}
Let $H$ be the $(d_1,1)$-ordering of the diamond graph. Then $R_<(H) \le 14$.
\label{dgva1}
\end{theorem}

\begin{proof}
Suppose that there exists a 2-coloring $G$ containing no monochromatic copies of $H$. Define $V_{RR}, V_{RB}, V_{BB}, V_{BR}$ to be the sets of vertices in $G$ connecting to vertex 1 by red, red, blue, and blue edges and to vertex 2 by red, blue, blue, and red edges, respectively.

Note that all edges between pairs of vertices in $V_{RR}$ must be blue. If there exists a red edge between vertices $v_1,v_2 \in V_{RR}$, then vertices $1$, $v_1$, $v_2$, and $2$ must form a copy of $H$. Since $V_{RR}$'s edges are all blue, $|V_{RR}| \le 3$. By symmetric reasoning, $V_{BB}$ contains only red edges, and $|V_{BB}| \le 3$. If $|V_{RB} \cup V_{BR}| \ge 6$, then since $R(K_3) = 6$, there must exist vertices $v_1,v_2,v_3 \in |B \cup C|$ forming a monochromatic copy of $K_3$. Consider the following cases:
\begin{itemize}
\item $v_1,v_2,v_3$ all belong to only one of $V_{RB}$ or $V_{BR}$: Without loss of generality, $v_1,v_2,v_3 \in V_{RB}$. If the three vertices form a red copy of $K_3$, then vertices $1,v_1,v_2,v_3$ form a red copy of $K_4$. If the three vertices form a blue copy of $K_3$, then vertices $2,v_1,v_2,v_3$ form a blue copy of $K_4$.

\item $\{v_1,v_2,v_3\}$ is neither a subset of $V_{RB}$ or of $V_{BR}$: Without loss of generality, assume $v_1,v_2 \in V_{RB}$ and $v_3 \in V_{BR}$. If $v_1,v_2,v_3$ form a red copy of $K_3$, then vertices $1,v_1,v_2,v_3$ form a red copy of $H$. If they form a blue copy of $K_3$, then vertices $2,v_1,v_2,v_3$ form a blue copy of $H$.
\end{itemize}
Therefore, $|V_{RB} \cup V_{BR}| = |V_{RB}| + |V_{BR}| \le 5$. Combining these results,
$|G| = 2 + |V_{RR}| + |V_{RB}| + |V_{BR}| + |V_{BB}|
\le 2 + 3 + 5 + 3
\le 13$.
\end{proof}

\begin{theorem}
Let $H$ be a $(d_1,2)$-ordering of the diamond graph. Then $R_<(H) \le 16$.
\label{dgva2}
\end{theorem}

\begin{proof}
Let $G$ be an ordered 2-coloring avoiding $H$. Without loss of generality, edge $(1,2)$ is red. Define $V_{RR}, V_{RB}, V_{BB}, V_{BR}$ to be the sets of vertices in $G$ connecting to vertex 1 by red, red, blue, and blue edges and to vertex 2 by red, blue, blue, and red edges, respectively.

If $V_{BR}$ contains a monochromatic triangle, then the vertices of the triangle form either a blue copy of $K_4$ with vertex 1 or a red copy of $K_4$ with vertex 2. Thus, $V_{BR}$ must avoid triangles. This means that $|V_{BR}| \le R(K_3) - 1 = 5$. By similar reasoning, $V_{RB}$ avoids triangles, and $|V_{RB}| \le 5$.

If $V_{BB}$ contains a blue edge between some two vertices $u$ and $v$, then vertices 2, $u$, $v$, and 1 form a copy of $H$. Thus, all edges between vertices in $V_{BB}$ must be red, implying that $|V_{BB}| \le 3$.

By similar reasoning, all edges in $V_{RR}$ must be blue, and $|V_{RR}| \le 3$. Moreover, if there is a red edge between vertices $u \in V_{RB}$ and $v \in V_{RR}$, then vertices 2, $v$, 1, and $u$ form a copy of $H$. This means that all edges between $V_{RB}$ and $V_{RR}$ are blue. Now note that $|G| \le 15$ in all possible cases:

\begin{itemize}
\item $|V_{RB}| = 0$: We get $|G| = 2 + |V_{BB}| + |V_{RB}| + |V_{RR}| + |V_{BR}| \le 2 + 3 + 0 + 3 + 5 = 13$.

\item $|V_{RB}| \ge 1, |V_{RR}| = 0$: We get $|G| = 2 + |V_{BB}| + |V_{RB}| + |V_{RR}| + |V_{BR}| \le 2 + 3 + 5 + 0 + 5 = 15$.

\item $|V_{RB}| \ge 1, |V_{RR}| = 1$: Suppose for the sake of contradiction that $|V_{RB}| + |V_{BR}| \ge 10$. Since $|V_{RB}| \le 5$ and $|V_{BR}| \le 5$, this can only be true when $|V_{RB}| = |V_{BR}| = 5$. Define a \emph{good triangle} as a monochromatic triangle containing one vertex in either $V_{RB}$ or $V_{BR}$ and the other two vertices in $V_{BR}$ or $V_{RB}$, respectively, both of which have order-label greater than does the single vertex. Let $v$ be the smallest vertex in $V_{RB} \cup V_{BR}$. If $v \in V_{RB}$, then by the Pigeonhole Principle, at least three edges between $v$ and vertices in $V_{BR}$ are the same color, blue without loss of generality. Let the three vertices be $u_1$, $u_2$, and $u_3$. Since $V_{BR}$ avoids triangles, at least one of the edges connecting $u_1$, $u_2$, and $u_3$ must be blue. Without loss of generality, the edge $(u_1,u_2)$ is blue. Then $v$, $u_1$, and $u_2$ form a good triangle. By similar reasoning, if $v \in V_{BR}$, there must be a good triangle.

Note that the existence of a good triangle implies the existence of a copy of $H$ formed by the good triangle and one of vertices 1 or 2. Therefore, we have established $|V_{RB}| + |V_{RB}| \le 9$, implying that $|G| = 2 + |V_{BB}| + |V_{RR}| + (|V_{RB}| + |V_{BR}|) \le 2 + 3 + 1 + 9 = 15$.

\item $|V_{RB}| \ge 1, |V_{RR}| \ge 2$: Since $V_{RR}$ contains only blue edges and edges between $V_{RB}$ and $V_{RR}$ are all blue, $V_{RR}$ must avoid triangles, lest $G$ contains a blue $K_4$. Thus, $|V_{RR}| = 2$. Moreover, if $V_{RB}$ contains a blue edge, then the four vertices comprising the blue edge in $V_{RB}$ and the elements of $V_{RR}$ form a blue copy of $K_4$. Thus, all edges in $V_{RB}$ must be red. Since $V_{RB}$ must avoid red triangles, $|V_{RB}| \le 2$. Hence $|G| = 2 + |V_{BB}| + |V_{RB}| + |V_{RR}| + |V_{BR}| \le 2 + 3 + 2 + 2 + 5 = 14$.
\end{itemize}
\end{proof}

\begin{theorem}
Let $H$ be the $(d_2,1)$-ordering of the diamond graph. Then $R_<(H) \le 13$.
\label{dgvb1}
\end{theorem}

\begin{proof}
Suppose that there exists a 2-coloring $G$ avoiding $H$. Without loss of generality, edge $(1,2)$ is red. Define $V_{RR}, V_{RB}, V_{BB}, V_{BR}$ to be the sets of vertices in $G$ connecting to vertex 1 by red, red, blue, and blue edges and to vertex 2 by red, blue, blue, and red edges, respectively.

If $|V_{RR}| \ge 2$, then there exist two vertices $v_1,v_2$ both connected to vertices $1$ and $2$ by red edges. Note that the four vertices $v_1,1,2,v_2$ form a red copy of $H$. It follows that $|V_{RR}| \le 1$. If $|V_{RB}| \ge 3$, then there exist three vertices $v_1,v_2,v_3 \in V_{RB}$ such that edges $(v_1,v_2)$ and $(v_2,v_3)$ are the same color. If the two edges are red, then vertices $v_1, 1, v_3, v_2$ form a red copy of $H$. If the two edges are blue, then vertices $v_1,2,v_2,v_3$ form a blue copy of $H$. It follows that $|V_{RB}| \le 2$. By the same argument, except with red and blue swapping roles, we have that $|V_{BR}| \le 2$. 

Suppose that $|V_{BB}| \ge 6$. If there exist three vertices $v_1,v_2,v_3 \in V_{BB}$ such that edges $(v_1,v_2)$ and $(v_2,v_3)$ are both blue, then vertices $v_1, 1,v_3,v_2$ form a blue copy of $H$. Thus no two blue edges in $V_{BB}$ share an endpoint, implying that each blue connected component contains at most two vertices. Let $v_1$ be the smallest vertex in $V_{BB}$, and select $v_2, v_3, v_4$ to not be in $v_1$'s blue connected component. Since $v_2, v_3, v_4$ cannot all be in the same blue connected component, without loss of generality, $v_4$ is in a different connected component from both $v_2$ and $v_3$. Thus $(v_1,v_2)$, $(v_1,v_3)$, $(v_1,v_4)$, $(v_2,v_4)$, and $(v_3,v_4)$ are all red, vertices $v_2,v_1,v_3,v_4$ form a red copy of $H$. Hence, $|V_{BB}| \le 5$.

Combining these results,
$|G| = 2 + |V_{RR}| + |V_{RB}| + |V_{BR}| + |V_{BB}|
\le 2 + 1 + 2 + 2 + 5
\le 12$.
\end{proof}

\begin{theorem}
Let $H$ be a $(d_2,2)$-ordering of the diamond graph. Then $R_<(H) \le 17$.
\label{dgvb2}
\end{theorem}

\begin{proof}
Let $G$ be an ordered 2-coloring on $n$ vertices avoiding $H$. Without loss of generality, edge $(1,n)$ is red. Define $V_{RR}, V_{RB}, V_{BB}, V_{BR}$ as the sets of vertices in $G$ connecting to vertex 1 by red, red, blue, and blue edges and to vertex $n$ by red, blue, blue, and red edges, respectively.

Note that if $V_{BB}$ contains any blue edges, then there exists a copy of $H$ in $G$. Thus, $V_{BB}$ must contain only red edges. In order for $V_{BB}$ to avoid $K_4$, we get that $|V_{BB}| \le 3$. By similar reasoning, all edges in $V_{RR}$ must be blue, and $|V_{RR}| \le 3$.

Let $P_{3a}$ be the ordered 3-vertex path 1---2---3 and $P_{3b}$ be the ordered path 2---1---3. Note that $V_{RB}$ cannot contain a blue copy of $P_{3a}$ or a red copy of $P_{3b}$, since such copies could be combined with vertices $n$ or $1$, respectively, in order to obtain copies of $H$. We will use this to show that $|V_{RB}| \le 4$.

Suppose that there exists an ordered 2-coloring on 5 vertices avoiding red copies of $P_{3a}$ and blue copies of $P_{3b}$. This 2-coloring must, by extension, avoid monochromatic copies of $K_3$. The only 2-colorings on 5 vertices avoiding $K_3$ are those in which the set of red edges and the set of blue edges each form a 5-vertex cycle. Let the blue edges in our 2-coloring be $(v_1,v_2),(v_2,v_3),(v_3,v_4),(v_4,v_5),$ and $(v_5,v_1)$. Without loss of generality, $v_1 < v_2$. Since $V_{RB}$ avoids blue copies of $P_{3a}$, no path of three vertices can be in increasing or decreasing order, implying that $v_2 > v_3$, $v_3 < v_4$, $v_4 > v_5$, and $v_5 < v_1$. However, this means that $v_5$, $v_1$, and $v_2$ form a blue copy of $P_{3a}$, a contradiction. Thus, $|V_{RB}| \le 4$.

Similarly, observe that $V_{BR}$ cannot contain a red copy of $P_{3a}$ or a blue copy of $P_{3b}$. By duplicating the above argument, we have $|V_{BR}| \le 4$.

We have
$|G| = 2 + |V_{BB}| + |V_{RB}| + |V_{RR}| + |V_{BR}|
\le 2 + 3 + 4 + 3 + 4
= 16$.
\end{proof}

The remaining results in this section prove upper bounds for the not-yet-discussed $1$-orderings of the $3$-pan. These results are proven using variants of classic single-vertex anchoring.

\begin{theorem}
Let $H$ be a $(e_1,2)$-ordering of the 3-pan. Then $R_<(H) \le 10$.
\label{3pva2}
\end{theorem}

\begin{proof}
Let $G$ be an ordered 2-coloring avoiding $H$. Without loss of generality, the edge between vertices 1 and 2 in $G$ is red. Let $A$ be the set of vertices in $G\backslash \{2\}$ connected to vertex 1 by a red edge, and let $B$ be the set of vertices in $G$ that are connected to vertex 1 by a blue edge.

Observe that $A$ cannot contain a red edge, since vertices $2$ and $1$ could be combined with such an edge to form a copy of $H$. Therefore, in order for $A$ to avoid $K_4$, we get $|A| \le 3$. Now let $u$ be the smallest vertex in $B$, and define the set $C$ as $B \backslash \{u\}$. Observe that $C$ must avoid blue edges, since vertices $u$ and $1$ could be combined with such an edge to form a copy of $H$. Thus $C$ has size at most $3$. Combining our results, we have
$|G| = 3 + |A| + |C|
\le 3 + 3 + 3
= 9$.
\end{proof}

\begin{theorem}
Let $H$ be the $(e_2,1)$-ordering of the 3-pan. Then $R_<(H) = 7$.
\label{3pvb1}
\end{theorem}

\begin{proof}
Let $G$ be an ordered 2-coloring avoiding $H$. For the sake of contradiction, suppose that $|G| = 7$. Let $V_R$ be the set of vertices in $G$ connected to vertex 1 by red edges, and let $V_B$ be the set of vertices connected to vertex 1 by blue edges. Since $|G| = 1 + |V_R| + |V_B|$, we know that $|V_R| + |V_B| = 6$. Without loss of generality, $|V_R| \ge |V_B|$. Consider the following cases:
\begin{itemize}
\item $|V_R| > |V_B|$: Since $|V_R| + |V_B| = 6$, $V_R$ contains at least 4 vertices. If $V_R$ contains a red edge $(v_1,v_2)$, then for any vertex $v_3 \in V_R \setminus \{v_1,v_2\}$, vertices $v_3,1,v_1,v_2$ form a red copy of $H$, a contradiction. Otherwise, if $V_R$ contains only blue edges, then since $|V_R| \ge 4$, there exists a blue copy of $K_4$ and, therefore, a blue copy of $H$ in $V_R$, a contradiction.

\item $|V_R| = |V_B| = 3$: As before, if $V_R$ contains a red edge $(v_1,v_2)$, then for any vertex $v_3 \in V_R \setminus \{v_1,v_2\}$, vertices $v_3,1,v_1,v_2$ form a red copy of $H$, a contradiction. Thus, assume that $V_R$ contains only blue edges. Similarly, we can assume that $V_B$ contains only red edges. Without loss of generality, vertex 2 is in $V_R$ and $V_R = \{2,v_1,v_2\}$ for some $v_1, v_2$. Consider the 3 edges between vertex 2 and the vertices in $V_B$. If there exists a blue edge $(2,u)$, where $u \in V_B$, then vertices $u,2,v_1,v_2$ form a blue copy of $H$, a contradiction. Thus, all edges between vertex 2 and the vertices in $V_B$ are red. It follows that vertex 2 and the three vertices in $V_B$ form a red copy of $K_4$, a contradiction.
\end{itemize}

We have shown that all ordered 2-colorings on at least 7 vertices must contain a copy of $H$. Since the Ramsey number of the 3-pan is 7 (Page 391 of \cite{ChHa72ii}), $R_<(H) = 7$.
\end{proof}

\begin{theorem}
Let $H$ be a $(e_2,2)$-ordering of the 3-pan. Then $R_<(H) \le 10$.
\label{3pvb2}
\end{theorem}

\begin{proof}
Let $G$ be an ordered 2-coloring avoiding $H$. Without loss of generality, $(1, 2)$ is blue.

Let $A$ be the set of vertices in $G \setminus \{1\}$ connected to vertex 2 by a blue edge. We will show that $A$ has size at most 2. First, observe that $A$ contains only red edges. Indeed, if $A$ contained a blue edge $(u,v)$, then the vertices 1, 2, $u$, and $v$ would form a blue copy of $H$, a contradiction. Suppose that $A$ contains at least 3 vertices $u < v < w$. Since $u$, $v$, and $w$ form a red triangle, in order for vertices 1, $u$, $v$, and $w$ not to form a red copy of $H$, the edge $(1,u)$ must be blue. However, this means that vertices $v$, 2, 1, and $u$ form a blue copy of $H$, a contradiction. Thus, $|A| \le 2$.

Let $G'$ comprise the vertices in $G \setminus (\{1, 2\} \cup A)$. Let $v_1$ be the smallest vertex in $G'$, and let $B$ be the subset of $G'$ connected to $v_1$ by red edges. Since all vertices in $B$ are connected to vertices 2 and $v_1$ by red edges and are all greater than 2 and $v_1$, any two vertices in $B$ can be combined with 2 and $v_1$ to yield a red copy of $H$. Thus, $|B| \le 1$.

Let $G''$ comprise the vertices in $G \setminus (\{1, 2, v_1\} \cup A \cup B)$. Let $v_2$ be the smallest vertex in $G''$. Let $C$ be the subset of $G''$ connected to $v_2$ by blue edges, and let $D$ be the subset of $G''$ connected to $v_2$ by red edges. Since $v_1$ and $v_2$ connect to vertices in $C$ only by blue edges, any two vertices in $C$ can be combined with $v_1$ and $v_2$ to yield a blue copy of $H$. Thus, $|C| \le 1$. Similarly, any two vertices in $D$ can be combined with 2 and $v_2$ to yield a red copy of $H$, forcing that $|D| \le 1$. Consequently,
$|G| = 4 + |A| + |B| + |C| + |D| \le 4 + 2 + 1 + 1 + 1 = 9$.
\end{proof}

\section{Infinite Families of Graphs}\label{sec:infinite}

In this section, we prove upper bounds on two infinite families of ordered Ramsey numbers. Theorem \ref{cycpath} says that for all $1$-orderings $H$ of $n$-vertex paths, $R_<(H) \in O(n)$. Theorem \ref{inffam} proves a more general result. For any graph $H$ containing a vertex $v$ with order-label $1$, Theorem \ref{inffam} bounds $R_<(K_n, H)$ in terms of $R_<(K_m, H \setminus \{v\})$ for $m \le n$. Consequently, Theorem \ref{inffam} can sometimes be used to reduce ordered Ramsey number upper bounds to classical Ramsey number results.

Our first theorem applies the parent-graph technique to any $1$-ordering of a path. This naturally extends our argument for $P_4$ (Theorem \ref{4p}).

\begin{theorem}
Let $H$ be a 1-ordering of the path $P_n$. Then $R_<(H) \le R(C_n)$, where $C_n$ is the $n$-cycle.
\label{cycpath}
\end{theorem}

\begin{proof}
Let $G$ be an ordered 2-coloring consisting of $R(C_n)$ vertices. By definition, $G$ must contain a monochromatic copy of $C_n$. Let $m$ be the order-label of the ordered vertex in $H$. Consider the monochromatic $n$-path formed by the vertices of the copy of $C_n$ in which the $m$-th smallest vertex in the cycle has the same position as the ordered-vertex in $H$. Since $H$ only has 1 ordered vertex, this monochromatic $n$-path is a copy of $H$. It follows that $G$ contains $H$.
\end{proof}

Theorem \ref{cycpath} provides a linear upper bound on the ordered Ramsey numbers of 1-orderings of paths.

\begin{corollary}
For any 1-ordering $H$ of the path $P_n$, where $n \ge 5$, $R_<(H) \le 2n - 1$ for odd $n$, and $R_<(H) \le \frac{3}{2}n - 1$ for even $n$.
\end{corollary}

\begin{proof}
Since $R(C_n) + 1$ is $2n$ for odd $n>4$ and $\frac{3}{2}n$ for even $n > 4$ \cite{BoEr73, FaSc74},  Theorem \ref{cycpath} suffices.
\end{proof}

Next, we present upper bounds for a more general family of ordered Ramsey numbers. In particular, the following bound can be applied to $R_<(K_n, H)$ for any $k$-ordered graph $H$ containing some vertex with order-label $1$. The proof of our bound relies on a single-vertex anchoring argument.

\begin{definition}
Let $H$ be a $k$-ordered graph for any $k \ge 0$. Define \emph{$H^+$} as the $(k+1)$-ordered graph with vertex set $H \cup {v}$, where $v$ has order-label 1 and all order-labels in $H$ are increased 1, and with edge set $H \cup \{(v,u):u\in H\}$.
\end{definition}

\begin{theorem}
For any $k$-ordered graph $H$, we have $R_<(K_n,H^+) \le \sum_{i = 1}^n R_<(K_i,H) - n + 1.$
\label{inffam}
\end{theorem}

\begin{proof}
As a base case, suppose that $n = 1$. Then $R_<(K_1,H^+) = 1$, which is $\le R_<(K_1,H)$, as claimed. 
Proceed by induction on $n \ge 2$, and suppose that $R_<(K_{n - 1},H^+) \le \sum_{i = 1}^{n - 1} R_<(K_i,H) - n + 2$. Let $G$ be an ordered 2-coloring that does not contain a red copy of $K_n$ or a blue copy of $H^+$. Let $A$ be the set of vertices in $H$ connected to vertex 1 by red edges, and let $B$ be the set of vertices in $H$ connected to vertex 1 by blue edges. Note that $A$ cannot contain a red copy of $K_{n - 1}$ or a blue copy of $H^+$ and that $B$ cannot contain a red copy of $K_n$ or a blue copy of $H$. This means that $|A| \le R_<(K_{n - 1}, H^+) - 1 \le \sum_{i = 1}^{n - 1} R_<(K_i,H) - n + 1$, by our inductive hypothesis, and that $|B| \le R_<(K_n,H) - 1$. Consequently,
$|G| = |A| + |B| + 1
\le \sum_{i = 1}^n R_<(K_i,H) - n + 1$, as claimed.
\end{proof}

The following corollary uses Theorem \ref{inffam} to generalize our bound on the $(d_2,1)$-ordering of the diamond graph, which can be regarded simply as $P_3^+$. Note that this corollary gives an upper bound for $R_<(P_n^+)$ since $R_<(P_n^+) \le R_<(K_{n + 1},P_n^+) \le \frac{1}{2}(n^3 - n) + 1$.

\begin{corollary}
For positive integers $m,n \ge 2$, $R_<(K_m,P_n^+) \le (n - 1)m(m - 1)/2 + 1$.
\label{corPplus}
\end{corollary}

\begin{proof}
It is well known that the Ramsey number of a complete graph on $m$ vertices and a tree on $n$ vertices is $(m - 1)(n - 1) + 1$ \cite{Chvatal77}. Since $P_n$ is a tree, combining this formula with Theorem \ref{inffam} yields
\begin{equation*}
    \begin{split}
R_<(K_m,P_n^+) & \le \sum_{i = 1}^m R_<(K_i,P_n) - n + 1 \\
& = \sum_{i = 1}^m ((i - 1)(n - 1) + 1) - n + 1 \\
& = (n - 1)\sum_{i = 0}^{m - 1} i + m - n + 1 \\
& = (n - 1)\frac{m(m - 1)}{2} + 1.
\end{split}
\end{equation*}
\end{proof}

Note that Theorem \ref{inffam} is not restricted to bounds for $1$-orderings. For example, one can generalize Corollary \ref{corPplus} by applying Theorem \ref{inffam} repeatedly to $P_n^+$ in order to establish that $R_<(K_m, P_n^{++\cdots +}) \in O(m^{1 + j}n)$, where the $^+$ operator is iterated $j$ times.

\section{Conclusion and Future Work}\label{sec:conc}

In this paper, we established upper bounds for ordered Ramsey numbers of all $1$-orderings on 4 vertices. Additionally, we found upper bounds for several infinite families of orderings, including an upper bound for $R_<(K_n, H^+)$ for all $k$-orderings $H$ with a vertex of order-label $1$ (Theorem \ref{inffam}). Our proofs suggest several new paradigms for finding upper bounds for ordered Ramsey numbers of small graphs. In particular, the following two ideas have proven to be effective: (1) exploiting the existance of an unordered graph that is guaranteed to contain a copy of our ordered graph, and (2) analysing 2-colorings from the perspective of two anchor vertices.

Since our research so far focuses almost exclusively on upper bounds on ordered Ramsey numbers, one important direction of future research will be to find lower bounds. In some cases, finding lower bounds may not be difficult. For example, our upper bounds for the $1$-orderings of the 4-path are trivially tight, since they match the unordered Ramsey number. In other cases, classical Ramsey number lower bounds \cite{BoEr73, ChHa72ii, Exoo89, FaSc74, Spencer75} may be easily adapted to the ordered context.

Finding a lower bound on a Ramsey number of a graph entails finding a construction of a 2-coloring avoiding the graph. Due to the computational complexity of examining all 2-colorings of $K_n$, of which there are $2^{\Theta(n^2)}$, this construction usually cannot be found solely by means of a brute-force search with a computer program. In some cases, however, it may be possible to find a construction with assistance from a computer program. For example, as a preliminary result, we have worked to find a lower bound for $R_<(DG)$, where $DG$ is the ordering of the diamond graph assigning values $1, 2, 3, 4$ to vertices $d_1, d_2, d_3, d_4$, respectively. By constructing a skeleton for a lower bound construction based on the proof of Theorem \ref{dgva1}, and then filling in the remaining edges by means of brute force, we are able to show that $R_<(DG) \ge 12$. Figure \ref{figlower} provides an example of a 2-coloring which establishes this bound. Interestingly, we were able to generate 25536 constructions for our lower bound, indicating that our result is likely not tight.

Additionally, the following directions of future work seem interesting.
\begin{enumerate}
\item Since ordered Ramsey numbers are always at least as large as Ramsey numbers, it is surprising that $R_<(G) = R(G)$ when $G$ is any $1$-ordering of the $4$-path (Theorem \ref{4p}), as well as for three of the 1-orderings of the 4-star (Theorem \ref{4sva1}, Theorem \ref{4sva2}, and Theorem 12 of \cite{BCKK13}). Can one find an infinite family of ordered graphs (besides graphs such as $K_n$ and $C_n$ which contain as subgraphs all of their $1$-orderings) for which this is the case?
\item Our results focus on graphs with a single component. Given a graph $G$ with components $G_1, \ldots, G_k$, let $G'$ be a $1$-ordering of $G$. What can we say about $R_<(G')$ in terms of the ordered Ramsey numbers of orderings of $G_1, \ldots, G_k$?
\item Theorem \ref{cycpath} extends our upper bounds for ordered Ramsey numbers of 1-orderings of $P_4$ to upper bounds for 1-orderings of $P_n$. Are there other 4-vertex graphs for which the upper bounds can be extended to consider an infinite family of 1-orderings? 
\end{enumerate}

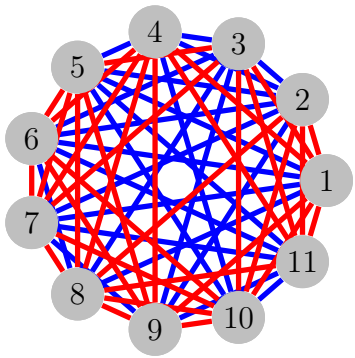
\begin{figure}[h]
\begin{center}
\begin{tikzpicture}
\tikzstyle{vertex}=[circle,fill=black!25,minimum size=20pt,inner sep=0pt];
\tikzstyle{edge} = [draw,line width = 2pt];
\foreach \pos/\name in {{(2,0)/1},{(1.68,1.08)/2},{(.83,1.82)/3},{(-.28,1.98)/4},{(-1.31,1.51)/5},{(-1.92,.56)/6},{(1.68,-1.08)/11},{(.83,-1.82)/10},{(-.28,-1.98)/9},{(-1.31,-1.51)/8},{(-1.92,-.56)/7}}, \node[vertex] (\name) at \pos {$\name$};
\foreach \source/\dest in {3/2,4/2,4/3,5/1,5/2,5/4,6/1,6/2,6/3,7/1,7/2,8/1,8/3,8/6,9/2,9/3,9/7,10/4,10/5,11/5,11/6,11/7,11/9,11/10} \path[edge][blue](\source) -- node[] {$$} (\dest);
\foreach \source/\dest in {2/1,3/1,4/1,5/3,6/4,6/5,7/3,7/4,7/5,7/6,8/2,8/4,8/5,8/7,9/1,9/4,9/5,9/6,9/8,10/1,10/2,10/3,10/6,10/7,10/8,10/9,11/1,11/2,11/3,11/4,11/8} \path[edge][red](\source) -- node[] {$$} (\dest);
\foreach \pos/\name in {{(2,0)/1},{(1.68,1.08)/2},{(.83,1.82)/3},{(-.28,1.98)/4},{(-1.31,1.51)/5},{(-1.92,.56)/6},{(1.68,-1.08)/11},{(.83,-1.82)/10},{(-.28,-1.98)/9},{(-1.31,-1.51)/8},{(-1.92,-.56)/7}}, \node[vertex] (\name) at \pos {$\name$};
\end{tikzpicture}
\caption{An ordered 2-coloring on 11 vertices which avoids DG.}
\label{figlower}
\end{center}
\end{figure}

\section{Acknowledgements}

The author would like to thank the MIT PRIMES program for providing him with the resources to conduct this research project. The author thanks William Kuszmaul for many useful conversations throughout the project, as well as for suggesting edits to the paper. His help was invaluable. The author also thanks Prof. Jacob Fox for suggesting an initial direction of research.

\clearpage

\pagestyle{empty}\singlespace

\end{document}